\documentclass{amsart}
\usepackage{amsmath}
\usepackage{amssymb}
\usepackage{graphicx}
\usepackage{hyperref}
\usepackage{booktabs}
\usepackage{multirow}
\hypersetup{
    colorlinks=true,
    linkcolor=blue,
    filecolor=magenta,      
    urlcolor=cyan,
    pdftitle={Overleaf Example},
    pdfpagemode=FullScreen,
    }
\newtheorem{theorem}{Theorem}[section]
\newtheorem{lemma}[theorem]{Lemma}
\newtheorem{proposition}[theorem]{Proposition}
\newtheorem{corollary}[theorem]{Corollary}

\theoremstyle{definition}
\newtheorem{definition}[theorem]{Definition}
\newtheorem{example}[theorem]{Example}

\theoremstyle{remark}
\newtheorem{remark}[theorem]{Remark}

\numberwithin{equation}{section}

\newenvironment{spmatrix}{\left ( \begin{smallmatrix}} {\end{smallmatrix}\right )}

\newcommand{\placket}[1]{\left({#1}\right)}
\newcommand{\vlacket}[1]{\left\{{#1}\right\}}
\newcommand{\abslr}[1]{\left|{#1}\right|}
\newcommand{\imaginary}[0]{\sqrt{-1}\,}

\newcommand{\centerbox}[2]{\parbox{#1}{\centering #2}}

\begin{document}

\title[Self-affinities of planar curves]{Self-affinities of planar curves: towards unified description of aesthetic curves}

\author{Shun Kumagai}
\address{Institute of Mathematics-for-Industry, Kyushu University, 744
Motooka, 819-0395, Fukuoka, Japan.}
\email{s-kumagai@imi.kyushu-u.ac.jp}

\author{Kenji Kajiwara}
\address{Institute of Mathematics-for-Industry, Kyushu University, 744
Motooka, 819-0395, Fukuoka, Japan.}
\email{kaji@imi.kyushu-u.ac.jp}

\subjclass[2020]{53A15, 93B51, 65D18}



\date{June, 22, 2024}


\keywords{industrial design, planar curves, self-affinity, quadratic curves, log-aesthetic curves, equiaffine geometry, similarity geometry.}

\begin{abstract}
In this paper, we consider the self-affinity of planar curves. 
It is regarded as an important property to characterize the log-aesthetic curves which have been studied as reference curves or guidelines for designing aesthetic shapes in CAD systems. 
We reformulate the two different self-affinities proposed in the development of log-aesthetic curves. 
We give a rigorous proof that one self-affinity actually characterizes log-aesthetic curves, while another one characterizes parabolas. 
We then propose a new self-affinity which, in equiaffine geometry, characterizes the constant curvature curves (the quadratic curves). 
It integrates the two self-affinities, by which constant curvature curves in similarity and equiaffine geometries are characterized in a unified manner. 
\end{abstract}

\maketitle



\section{Introduction}
In the field of computer aided design (CAD), control over the \textit{visual language} \cite{Kepes} such as the impressions received from the components and outlines of a shape is highly dependent on the expertise of designers. 
Using spline curves such as B\'{e}zier curves, B-Spline curves, and non-uniform rational B-spline (NURBS) curves, one can design shapes interactively in a way that is suitable for generating in CAD systems \cite{CAGD}.
To design visually desirable shapes in CAD systems, these basic tools require some sort of reference curves or guidelines.

In 1995, inspired by the analysis of curves appearing in the shapes of designed cars, Harada, Mori, and Sugiyama \cite{Harada1995} suggested that a sort of self-affinity (the Harada self-affinity, the HSA) is important to characterize \textit{aesthetic} shapes. 
They formulated it by the linearity of the logarithmic curvature histogram (LCH, also known as logarithmic distribution diagram of curvature, LDDC). 
Curves such as logarithmic spirals and clothoids, which have been classically considered beautiful, 
give linear LCHs indeed. 
Harada, Yoshimoto, and Moriyama \cite{Harada1999} classified planar curves into five types according to visual language in terms of LCH gradients. 
Curves sampled from several artifacts and natural structures were investigated using this classification \cite{Harada2001,Inoue2006}. 

In 2005, Miura \cite{Miura2005} reformulated the above self-affinity (the Miura self-affinity, the MSA) using the logarithmic curvature graph (LCG) \cite{Gobi2014} which is the continuous limit of LCH. 
He introduced log-aesthetic curves (LACs) \cite{Miura2006} as a class of curves whose LCG is a line of prescribed slope. 
It is generated by applying the fine-tuning method \cite{MiuraChengWang} to clothoids, by which they are deformed to curves with linear LCGs whose gradients are arbitrarily controlled. 
LACs have been studied as reference curves for designing shapes in CAD systems \cite{Yoshida2006,Wang2021,Graiff2023,Tsuchie2024}.
As other important characterizations, LACs are known as critical points of the fairing energy functional \cite{Inoguchi2018} and invariant curves of integrable evolution in similarity geometry \cite{Inoguchi2023}. 


In this paper, we consider characterizations of curves in terms of self-affinities which have not been dealt with mathematically.
We present rigorous proof of the claim \cite{Miura2006} that the MSA characterizes LACs. 
On the other hand, the HSA has not been studied well. 
Despite Harada's original discussion on the relationship between the HSA and the linearity of LCG, we show that it is not the case and that the HSA actually characterizes \textit{parabolas}. 
We recall that parabolas are zero curvature curves in \textit{equiaffine geomerty}, while special LACs, circles and logarithmic spirals are the zero curvature curves and the constant curvature curves, respectively, in \textit{similarity geometry}. 
In view of this, we propose a new \textit{extendable self-affinity (the ESA)} that integrates the HSA and the MSA in terms of geometries in Klein's Erlangen program \cite{Equiaffine}. 
The main theorems are stated as follows. 
\begin{theorem}(Theorem \ref{MSALAC})\label{main1}    A curve possesses the MSA if and only if it is either a circle, a line, or a LAC. 
\end{theorem}
\begin{theorem}(Theorem \ref{HSAparabola})\label{main2}    A curve possesses the HSA if and only if it is either a line or a parabola. 
\end{theorem}
\begin{theorem}(Theorem \ref{extSA})\label{main3}   In equiaffine geometry, a curve possesses the ESA if and only if it is a constant curvature curve (a quadratic curve; either a parabola, an ellipse, or a hyperbola). 
\end{theorem}
Theorem \ref{main3} generalizes Theorem \ref{main2} in terms of the ESA in {equiaffine geometry}. 
In the case of logarithmic spirals and circles, Theorem \ref{main1} implies that the ESA in {similarity geometry} is equivalent to having a constant curvature. 
In other words, the HSA and the MSA intersect as the ESA that characterize constant curvature curves in corresponding geometries.

This paper is organized as follows. 
In Section 2, we present the basics of planar curves that will be referred to. 
In Section 3, we introduce the HSA and the MSA and prove Theorem \ref{main1} and Theorem \ref{main2}. 
In Section 4, we define the ESA as a generalization of the MSA and the HSA, and prove Theorem \ref{main3}. 
Section 5 is devoted to some concluding remarks implicating the connection among the main results.

\section{Preliminaries}
\subsection{Basics on planar curves}
This subsection refers to \cite{Curve}. 
Throughout this paper, we consider a \textit{parametric planar curve} (simply, a \textit{curve}). 
It is a smooth function $\gamma(t):I\rightarrow \mathbb{C}$ on an interval $I=I_\gamma\subset \mathbb{R}$. 
Once a curve $\gamma$ is given, let us assume a fixed base point at $\eta=\eta_\gamma\in I$.     
In other words, we regard a curve as the triplet $(I, \gamma(t), \eta)$ as above. 
As an additional assumption, we impose a regularity on a curve $\gamma$ such that the derivative ${d\gamma(t)}/{dt}$ is non-vanishing. 
We identify the complex number field $\mathbb{C}$ with the plane $\mathbb{R}^2$ naturally. 
\begin{definition}
    A \textit{reparametrization} between two curves $\gamma_i:I_i\rightarrow \mathbb{C},\ i=1,2$ is a smooth homeomorphism $t:I_1\rightarrow I_2$ such that 
    \begin{enumerate}
        \item $t(\eta_{\gamma_1})=\eta_{\gamma_2}$ and 
        \item $\gamma_1(t_1)=\gamma_2\circ t(t_1)$ for any $t_1\in I_1$. 
    \end{enumerate}
    If a reparametrization $t:I_1\rightarrow I_2$ is given, 
    we shall denote 
    \begin{align*}
        \gamma_2(t_1)&:=\gamma_2\circ t(t_1)=\gamma_1(t_1),\ t_1\in I_1,\\
        \gamma_1(t_2)&:=\gamma_1\circ t^{-1}(t_2)=\gamma_2(t_2),\ t_2\in I_2.
    \end{align*}
    Remark that the inverse map of a reparametrization $t_2=t_2(t_1)$ is denoted by $t_1=t_2(t_1)$. 
\end{definition}

\begin{lemma}\label{arclength}
    For any curve $\gamma:I\rightarrow \mathbb{C}$, there uniquely exists a globally increasing reparametrization $s=s_\gamma:I\rightarrow J\subset\mathbb{R}$ 
    such that 
    \begin{enumerate}
        \item $s(\eta_\gamma)=0$,\\[-2mm]
        \item $|\frac{d\gamma(s)}{ds}|=1$, \\[-2mm]
        \item $\frac{d^2\gamma(s)}{ds^2}=\sqrt{-1}\,\kappa(s)\frac{d\gamma(s)}{ds}$,\quad  $\kappa(s)=\kappa_\gamma(s):=\frac{d}{ds}\mathrm{arg}\left(\frac{d\gamma(s)}{ds}\right)$.
    \end{enumerate}
\end{lemma}
We use the notation $s=s_\gamma$ for the above \textit{arc length parameterization} of a curve $\gamma$. 
%
We introduce the Euclidian frame $\Phi^E:=(\gamma_s,\imaginary\gamma_s)$. 
Then, the \textit{(Euclidian) curvature} $\kappa=\kappa_\gamma$ reproduces the input curve $\gamma$ in the following sense. 
\begin{proposition}[Fundamental theorem of curves]\label{FTC}
    For a given non-negative, smooth function $\kappa(s):I\rightarrow \mathbb{R}$, the Frenet formula 
    \begin{align}
        \Phi^E_s=\Phi^E\begin{pmatrix}
            0&-\kappa\\ \kappa&0
        \end{pmatrix}
    \end{align}
    has a unique solution $\gamma(s):I\rightarrow \mathbb{C}$ such that $\kappa_\gamma(s)=\kappa(s)$ up to the congruent transformation group $G^E:=\{z\mapsto Az+b\mid A\in O(2),\ b\in\mathbb{C}\}$. 
\end{proposition}

We call the reciprocal $\rho=\rho_\gamma=1/\kappa_\gamma$ the \textit{curvature radius} of a curve $\gamma$. 
As a consequence of Proposition \ref{FTC}, it follows that the curvature radius $\rho_\gamma(s_0)$ is the radius of the unique \textit{osculating circle} that approximates $\gamma(s)$ at $s=s_0$ in quadratic order. 
In the following, we denote $s$-differential by $(\cdot)'$. 
\begin{proposition}\label{deformation_circle}
    Let $\gamma(s):I\rightarrow \mathbb{C}$ be a curve. 
    Then, for any matrix $A\in GL(2,\mathbb{R})$, it follows that 
    \begin{equation}\label{eqrhoA}
        \rho_{A\gamma}(s)=\frac{|A\gamma'(s)|^3}{\mathrm{det}A}\rho_\gamma(s).
    \end{equation}
\end{proposition}
\begin{proof}
We use the formula \cite{Curve} of 
curvature radius
\begin{align}\label{curvature}
    \rho_{\gamma}(t)=\frac{|\gamma_t(t)|^3}{\det(\gamma_{t}(t), \gamma_{tt}(t))},
\end{align}
where $t$ is an arbitrary parameter. 
Let $t:=s_\gamma$ be the arc length parameter of $\gamma$, then we have $\rho_\gamma(t)=\det(\gamma_t,\gamma_{tt})$. 
Applying \eqref{curvature} to $A\gamma(t)$ yields 
\begin{align}
    \rho_{A\gamma}(t)=\frac{|A\gamma_t(t)|^3}{\det(A\gamma_{t}(t), A\gamma_{tt}(t))}.
\end{align}
Since $\det(A\gamma_t,A\gamma_{tt})=\det (A(\gamma_t,\gamma_t))=\det A \det (\gamma_t,\gamma_{tt})=\det A/\rho_\gamma$, we have
\begin{align}
    \rho_{A\gamma}(t)=\frac{|A\gamma_t(t)|^3}{\det A}\rho_\gamma(t),
\end{align}
which is \eqref{eqrhoA}. 
\end{proof}
    
\subsection{Logarithmic Curvature Histogram and Graph}
Let $\gamma(s):[0,s_{\text{all}}]\rightarrow\mathbb{C}$ be a curve where $s_\text{all}>0$ is the total length. 
We will consider the length histogram of $\gamma$ against the logarithmic curvature radius $X=\log\rho$. 
For fixed $M,N\in\mathbb{N}$, let $\{R_i\}_{i=1}^M$ be the $M$ subdivisions of the range of $X$ of equal length and $\{\rho_j\}_{j=1}^N$ be the curvature radius of $N$ division points on $\gamma$ with equal arc length. 
That is,
    \begin{align}
        R_i&:=\vlacket{x\in\mathbb R\mid \frac{i}{M}\leq \frac{x-\min X}{\max X-\min X}<\frac{i+1}{M}},\ i=0,...,M,\\
        \rho_j&:=\rho_\gamma\placket{\frac{j}{N}\cdot s_\text{all}},\ j=0,...,N. 
    \end{align}
    The \textit{logarithmic curvature histogram (LCH)} \cite{Harada1995, Harada1999} of $\gamma$ is the histogram $\Gamma_{M,N}(\gamma)$ defined by counting the logarithmic value 
    \begin{equation}
        Y_i=\log\frac{\Delta s_i}{\Delta X_i}    :=
        \log\frac{\#\vlacket{j\mid\log\rho_j\in R_i}\cdot s_{\text{all}}/N}{|R_i|},\ i=0,...,M,
    \end{equation}
    against each domain $R_i$ (or its representative $X_i:=\min R_i$) unless $\#\{j\}=0$.
We note that the idea of taking logarithmic coordinates can be observed in the area of \textit{allometry} \cite{allo} in natural structures. 

Harada et al.\ pointed out in \cite{Harada1995,Harada1999} that the LCHs of ``aesthetic" curves drawn by professional car designers and modelers, and the keyline curves of actual cars have a linear tendency. 
Based on this observation, they proposed the following property. 
\begin{definition}[the Harada self-affinity, see also 
Definition \ref{HSA}, and Figure 6 in \cite{Harada1999}]\label{HSA_original}
    A curve possesses \textit{the Harada self-affinity (the HSA)} if its arbitrary subcurve coincides with the image of an affine deformation of the whole curve. 
\end{definition}

We will show that the linearity of LCHs and the HSA are actually different; 
the linearity of LCHs should not be thought of as a self-affinity in the Euclidian plane of curves but that in the plane of LCHs. 
Miura \cite{Miura2006} proposed an alternate self-affinity (the Miura self-affinity, Definition \ref{MSA}) that is regarded as a self-affinity in the \textit{logarithmic curvature graph (LCG)} of $\gamma(s):[0,s_\text{all}]\rightarrow\mathbb{C}$ defined by 
\begin{equation}\label{LCG}
    \Gamma(\gamma):=\vlacket{(X,Y)=\placket{\log \rho(s),\log \abslr{\frac{ds}{d\log \rho(s)}}}\Big|\, s\in [0,s_\text{all}]}.
\end{equation}

We now show that the continuous limits of LCHs are LCGs. 
This fact is mentioned in \cite{Nakano2003} but we give a mathematically rigorous proof.   
For LCH, we define
\begin{equation}
    f_{M,N}(X)=
    \begin{cases}
        e^{Y_i} &\text{ if $X\in R_i$ and $Y_i\neq-\infty$,}\\
        0&\text{ otherwise.}
    \end{cases}
\end{equation}
Then, we have:
\begin{proposition}\label{LCHLCG}
    Let $\gamma$ be a curve such that $\rho(s)$ is smooth and monotonous.
    Then, the distribution $\mu_{M,N}(dX)=\sum f_{M,N}(X)\,dX$ of $\Gamma_{M,N}(\gamma)$ strongly converges to the distribution $\mu(dX)=e^YdX$ of $\Gamma(\gamma)$ as $M,N\rightarrow\infty$. 
    In particular, the LCH plot converges to the LCG plot pointwise almost everywhere as $M,N\rightarrow\infty$. 
\end{proposition}
\begin{proof} 
    By the assumption, there exists a reparametrization $s=s(X)$ of $\gamma(s)$. 
    The line element is given by $ds(dX)=|\frac{ds(X)}{dX}|\,dX=e^YdX=\mu(dX)$.  
    We show that the values of arbitrary $[a,b)\subset \mathbb{R}$ measured by $\mu_{M,N}(dX)$
    and $ds(dX)$ are asymptotically equal as $M,N\rightarrow\infty$. 
    
    For $X=a,b$, let $i_X,j_X$ be the largest integers less than $\frac{X-X_0}{X_M-X_0},\ \frac{N {s(X)}}{s_\text{all}}$, respectively. 
    We have $\mu_{M,N}([a,b))=\frac{(j_b-j_a-1)s_\text{all}}{N}$ by definition of LCH. 
    For any $\varepsilon>0$, one can take sufficiently large $M,N$ so that  
    \begin{align}
        i_a+1<i_b,\label{eN0}
        \\
        \frac{s_\text{all}}{N}<\frac{\varepsilon}{4M},\label{eN1}
        \\
        \frac{X_M-X_0}{M}\max_{\substack{X\in R_i, 0\leq i\leq M}}e^{Y(X)}
        <\frac{\varepsilon}{4}.\label{eN2}
    \end{align}
    Then, the error between $ds([a,b))$ and $\mu_{M,N}([a,b))$ is estimated by 
    \begin{align}\label{ab}
        \bigsqcup_{i=i_a+1}^{i_b}R_i\subset [&a,b)\subset \bigsqcup_{i=i_a}^{i_b+1}R_i.
        \end{align}
By applying $ds$ to \eqref{ab}, we have
        \begin{align}            
        -ds(R_{i_a})<ds([a,b))&-\sum_{i=i_a}^{i_b}ds(R_i)<ds(R_{i_b+1}). \label{dsRi}
        \end{align}
Applying $\mu_{M,N}$ to \eqref{ab} gives
        \begin{align}    \label{muRi}
        -\mu_{M,N}(R_{i_a})<\mu_{M,N}([a,b))&-\sum_{i=i_a}^{i_b}\mu_{M,N}(R_i)<\mu_{M,N}(R_{i_b+1}).
        \end{align}
Since each curve segment is of length $\frac{s_\text{all}}{N}$, we have 
        \begin{align}  
        -\frac{s_\text{all}}{N}<ds(R_i)&-\mu_{M,N}(R_i)<\frac{s_\text{all}}{N}, \label{muconti}
    \end{align}
    as shown in Figure \ref{fig:muconti}. 
\begin{figure}[h]\label{fig:muconti}
    \includegraphics[height=55mm]{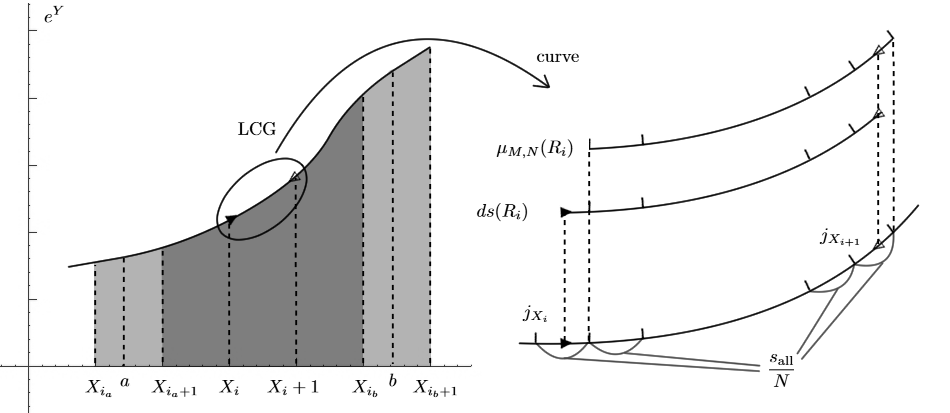}
    \caption{LCH, LCG, and curve: $\mu_{M,N}(R_i)$ counts the number of curve segments whose $X$-values at initial points belong to $R_i$. }
\end{figure}

    On the other hand, for any $i=0,...,M$, from \eqref{eN2} and we have
    \begin{align}\label{eN4}
        0\leq ds(R_i)<\frac{X_M-X_0}{M}\max_{\substack{X\in R_i,0\leq i\leq M}}     e^{Y(X)}<\frac{\varepsilon}{4}.
    \end{align}
    \eqref{eN1} and \eqref{muconti} yield
    \begin{align}\label{eN5}
        0\leq \mu_{M,N}(R_i)<ds(R_i)+\frac{s_\text{all}}{N}<\placket{\frac{1}{4}+\frac{1}{4M}}{\varepsilon}<\frac{2\varepsilon}{4}.
    \end{align}
    The triangle inequality yields
    \begin{align}\label{eN6}
        |ds([a,b))&-\mu_{M,N}([a,b))|\leq \ \\
        &+\sum_{i=i_a}^{i_b} \abslr{ds(R_i)-\mu_{M,N}(R_i)}+\abslr{\mu_{M,N}([a,b))-\sum_{i=i_a}^{i_b}\mu_{M,N}(R_i)}.\nonumber
    \end{align}
    We have from \eqref{eN6} by using \eqref{dsRi} and \eqref{muRi}     
    \begin{align}\label{eN7}
        \abslr{ds([a,b))-\mu_{M,N}([a,b))}< \ &\max_{i}ds(R_i)+\frac{Ms_\text{all}}{N}+\max_i\mu_{M,N}(R_i). 
    \end{align}
    Applying \eqref{eN1}, \eqref{eN4} and \eqref{eN5} to \eqref{eN7}, we conclude that 
    \begin{align}
        \abslr{ds([a,b))-\mu_{M,N}([a,b))}< \ &\frac{\varepsilon}{4}+\frac{\varepsilon}{4}+\frac{2\varepsilon}{4}=\varepsilon.
    \end{align}
    Thus we have a strong convergence. 
    The relation to graph plot refers to \cite{Scheffe1947}.
\end{proof}

For example, Figure \ref{fig:LCHparabola} shows LCHs and the LCG of a parabola. 
In general, the limit of LCH $\Gamma_{M,N}(\gamma)$ as $M,N\rightarrow\infty$ is regarded as the sum of LCG segments $\{\Gamma(\gamma|_{I_k})\mid I=\bigsqcup I_k,\ \rho_\gamma|_{I_k}$: monotonous$\}$.
\begin{figure}[h]\label{fig:LCHparabola}
    \includegraphics[height=120mm]{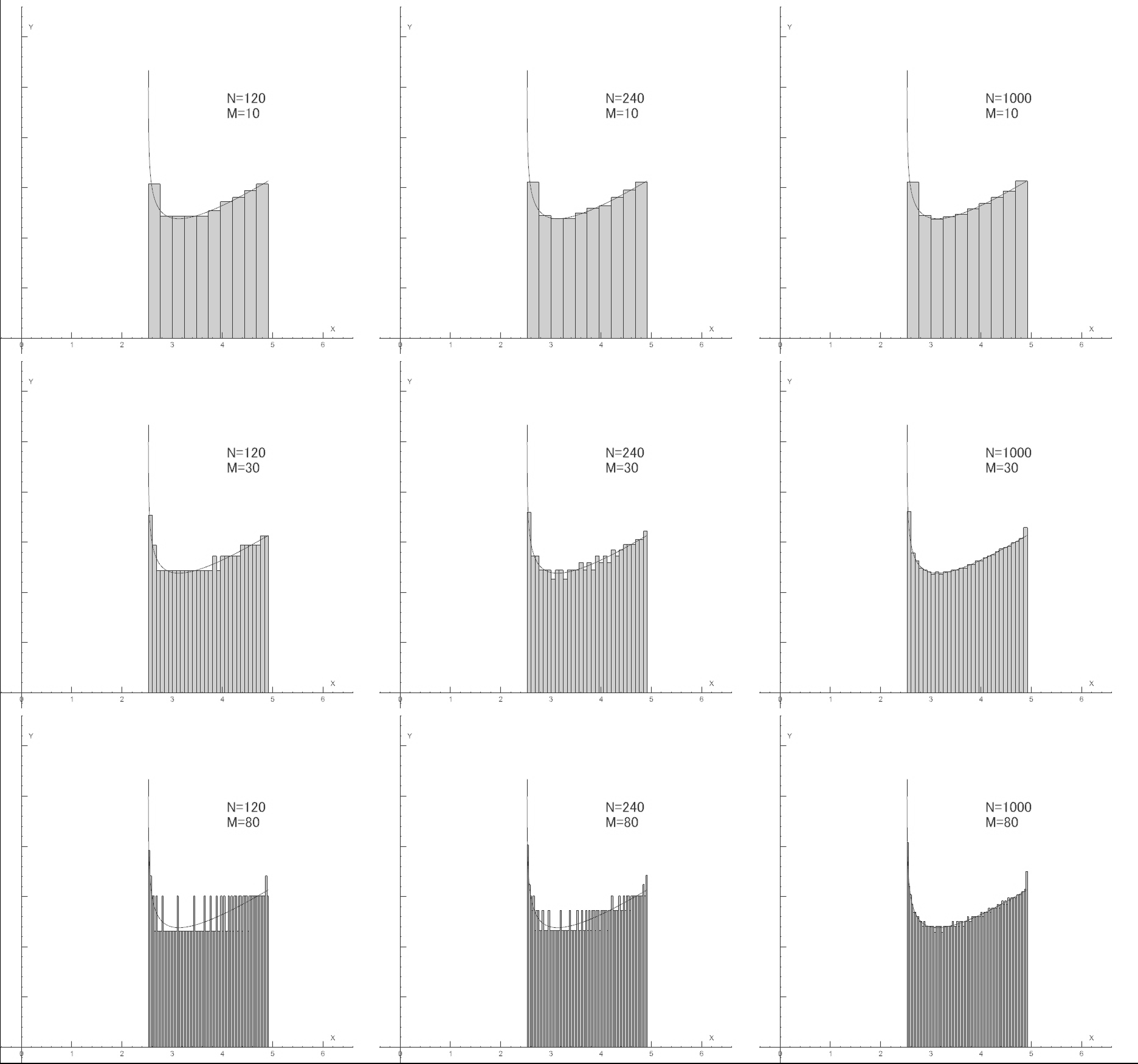}
    \caption{LCH and LCG of parabola $\gamma(t)=5t+\imaginary t^2$, $t\in [0,5]$ for $M=10,30,80$ and $N=120,240,1000$. LCG is represented by $X(t)=\frac{3}{2}\log(4t^2+25)-\log 10$, $Y(t)=\frac{1}{2}\log(4t^2+25)-\log\frac{12t}{4t^2+25}$.  }
\end{figure}

In the next section, we will discuss another self-affinity of curves characterizing the linearity of LCGs, and curves characterized by the Harada self-affinity. 
\section{The Miura and the Harada self-affinties}

\subsection{Log-Aesthetic Curve and the Miura self-Affinity}

Miura \cite{Miura2006} pointed out that a clothoid curve does not possess the HSA, while it has a linear LCG. 
He also defined the following class of curves with linear LCG constructed from the fine-tuning method \cite{MiuraChengWang}. 

\begin{definition}[Log-Aesthetic Curve]
    A \textit{log-aesthetic curve} (LAC) of slope $\alpha$ is a curve defined by
    \begin{align}\label{LACeq}
        \rho(s)=\left\{\begin{array}{ll}
             (\xi s+\eta)^\frac{1}{\alpha} & (\alpha\neq0), \\[2mm]
             e^{\xi s+\eta}&  (\alpha=0),
        \end{array}\right.
    \end{align}
    restricted to $\{s\mid \xi s+\eta\geq 0\}$, 
    where $\xi\in\mathbb{R}\setminus \{0\}$, $\alpha,\eta\in\mathbb{R}$.
    The equation (\ref{LACeq}) determines a unique curve up to congruent transformations by Proposition \ref{FTC}. 
\end{definition}
\begin{example} Figure \ref{fig:LAC} illustrates the following examples of LACs. 
    \begin{enumerate}
        \item A \textit{logarithmic spiral curve} $\gamma(t)=e^{(a+\imaginary b)t}$, $a+\imaginary b\in\mathbb{C}$: observe that
    \begin{equation*}
        s(t)=\sqrt{a^2+b^2}(e^{at}-1),\ \rho(t)=\frac{1}{b}\sqrt{a^2+b^2}e^{at}=\left(\frac{1}{b}s(t)+\frac{\sqrt{a^2+b^2}}{b}\right)^1. 
    \end{equation*}
    It is a LAC with $\alpha=1$. 
    \item A \textit{clothoid curve} $\gamma(t)=\int_{0}^te^{\imaginary at}dt$, $a\neq 0$: observe that
    \begin{equation*}
        s(t)=t,\ \rho(t)=\frac{1}{|2at|}=(2as(t)+0)^{-1}. 
    \end{equation*}
    It is a LAC with $\alpha=-1$. 
    \item A circle and also a line have constant curvatures.     
    They are regarded as the limit of a family of LACs as $\alpha\rightarrow\pm\infty$ \cite{Yoshida2006}. 
    Actually, for any constants $\xi,\eta,\rho_0\in\mathbb{R}$ with $(\xi,\eta)\neq (0,0)$, we have
    \begin{align}
        \lim_{\alpha\rightarrow \pm\infty }(\rho_0^\alpha(\xi s+\eta))^\frac{1}{\alpha}=\rho_0,
        \\
        \lim_{\alpha\rightarrow \pm\infty }\placket{\frac{\xi s+\eta}{\alpha^\alpha}}^\frac{1}{\alpha}=0. 
    \end{align}
    \end{enumerate}
    
\begin{figure}[h]\label{fig:LAC}
    \includegraphics[height=60mm]{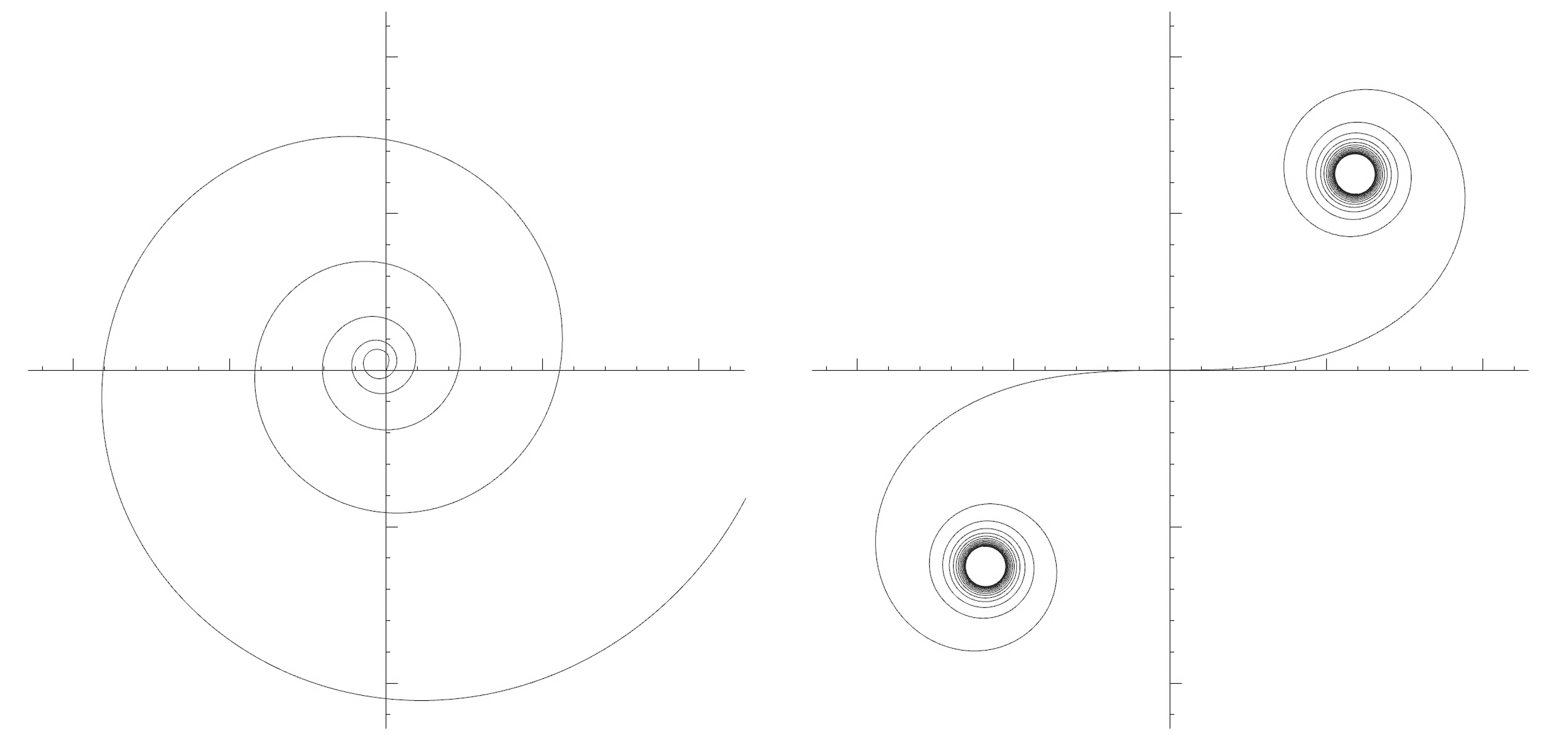}
    \caption{LAC: a logarithmic spiral ($\alpha=1$, left) and a clothoid curve ($\alpha=-1$, right).}
\end{figure}

One can see that the LCG gradient ${dY}/{dX}$ of a LAC is the constant $\alpha$ by \eqref{LACeq} and the formula \cite{Gobi2014} that follows from \eqref{LCG}:
\begin{equation}
        \frac{dY}{dX}(s)=1-\frac{\rho(s)\rho''(s)}{\rho(s)'^2}. \label{LCGgrad}
\end{equation}

\end{example}

We now discuss the self-affinity of LACs. 
We introduce the $\varepsilon$\textit{-shift} mapping which shifts the parameters of curves with their domains and base points shifted accordingly. Namely:  
\begin{definition}
    For any $\varepsilon>0$, we define the $\varepsilon${-shift} mapping $\Lambda_\varepsilon$ on the set of curves by 
    \begin{enumerate}
        \item $I_{\Lambda_\varepsilon\gamma}=I_\gamma-\varepsilon$, \\[-4mm]
        \item $\Lambda_\varepsilon\gamma(t)=\gamma(t+\varepsilon)$ for any $t\in I_{\Lambda_\varepsilon\gamma}$, \\[-4mm]
        \item $\eta_{\Lambda_\varepsilon\gamma}=\eta_\gamma+\varepsilon$,
    \end{enumerate}
     for each curve $\gamma=(I,\gamma(t),\eta)$. 
     We denote $\Lambda_\varepsilon F_\gamma:=F_{\Lambda_\varepsilon\gamma}$ for any function $F_\gamma$ of $\gamma$. 
\end{definition}
In particular, from the setting of arc length parametrization in Lemma \ref{arclength}, the $\varepsilon$-shift of $s$ yields
\begin{align}\label{ALshift}
    \Lambda_\varepsilon s(t)=s(t+\varepsilon)-s(\eta+\varepsilon). 
\end{align}
\begin{definition}[the Miura self-affinity]\label{MSA}
    We say that a curve $\gamma(s):I\rightarrow \mathbb{C}$ possesses \textit{the Miura self-affinity (the MSA)} if there exist $\mu,\nu>0$ and a reparametrization $t(s):I\rightarrow J$ such that for any $\varepsilon>0$, 
    \begin{equation}\label{eqMSA}
        \Lambda_\varepsilon(s_\gamma(t),\rho_\gamma(t))=(\mu^\varepsilon s_\gamma(t),\nu^\varepsilon\rho_\gamma(t)),\ \forall t\in J.
    \end{equation}
\end{definition}
\begin{remark}\label{geomSA}
    Definition \ref{MSA} implies that a curve $\gamma$ with the MSA has the following geometric property: take any subcurve $\gamma^1$. 
Let $\gamma_{a,b}^1$ be a curve obtained by applying arbitrary scale change 
$(\kappa(t),s(t))\mapsto (a\kappa(t),bs(t))$ to $\gamma^1$. 
Then there exists another subcurve $\gamma^2$ congruent to $\gamma^1_{a,b}$ by choosing $b=b(a,\gamma)$ appropriately.

Similarly, the geometric description of the HSA can be stated as follows: take any subcurve $\gamma^1$. 
Let $\gamma_{a,b}^1$ be a curve obtained by applying arbitrary scale change $(\mathrm{Re}\,\gamma^1(t),\mathrm{Im}\,\gamma^1(t))\mapsto (a\,\mathrm{Re}\,\gamma^1(t),b\,\mathrm{Im}\,\gamma^1(t))$. 
Then there exists another subcurve $\gamma^2$ affine equivalent to $\gamma^1_{a,b}$.
\end{remark} 
\begin{remark}
    Note that \eqref{eqMSA} defined by using the map $\Lambda_\varepsilon$ holds for specific parametrization $t(s)$. 
    For example, we will next show that a logarithmic spiral, whose curvature radius is given by $\rho(s)=\xi s+\eta$, possesses the MSA. 
    However, \eqref{eqMSA} does not hold for $t(s)=s$. 
    In fact, the $\varepsilon$-shift of just $s$ yields the following equation different from \eqref{eqMSA}: 
    \begin{align}
        \Lambda_\varepsilon(s,\rho(s))=((s+\varepsilon)-(0+\varepsilon),\xi(s+\varepsilon)+\eta)=(s,\rho(s)+\xi \varepsilon). 
    \end{align}
    The appropriate parametrization will be demonstrated in the proof of Theorem \ref{MSALAC}.
    \end{remark}

\begin{theorem}\label{MSALAC}
    A curve $\gamma:I\rightarrow \mathbb{C}$ possesses the MSA if and only if $\gamma$ is it is either a circle, a line, or a LAC. 
\end{theorem}
\begin{proof}
    First, we consider a LAC with $\alpha\neq0$, $\rho(s)=(\xi s+\eta)^\frac{1}{\alpha}$. 
    As mentioned in \cite{Miura2006}, we take a reparameterization $t$ so that $s=\frac{\eta}{\xi}(e^{\beta t}-1)$ for an arbitrary fixed constant $\beta \neq0$. 
    Then, for any $\varepsilon>0$, we have 
    \begin{align}
        \Lambda_\varepsilon s(t)&=s(t+\varepsilon)-s(\varepsilon)=\frac{\eta}{\xi}(e^{\beta (t+\varepsilon)}-e^{\beta \varepsilon})=e^{\beta\varepsilon}s(t).
    \end{align}
    Also, \eqref{LACeq} implies that 
    \begin{align}
       \Lambda_\varepsilon \rho(t)&=\rho(t+\varepsilon)=(\xi s(t+\varepsilon) +\eta)^\frac{1}{\alpha}=(\xi\frac{\eta}{\xi}(e^{\beta(t+\varepsilon)}-1)+\eta)^\frac{1}{\alpha}=e^{\frac{\beta}{\alpha}\varepsilon}\rho(t). 
    \end{align}    
    Thus the curve posseses the MSA with $\mu=e^{\beta}$ and $\nu=e^\frac{\beta}{\alpha}$. 
    
    Second, we consider a LAC with $\alpha=0$, $\rho(s)=e^{\xi s+\eta}$. 
    We take a reparameterization $t=\frac{\beta}{\xi} s$ for an arbitrary fixed $\beta \neq0$. 
    One can easily check that $\Lambda_\varepsilon s(t)=s(t)$ and $\Lambda_\varepsilon \rho(t)=e^{\xi\beta\varepsilon}\rho(t)$, which imply the MSA with $\mu=1$ and $\nu=e^{\beta}$. 
    
    Third, for a circle $\gamma(s)=\rho_0 e^{\imaginary{s}/{\rho_0}}$, take a reparametrization $t$ so that $s(t)=C(e^{\beta t}-1)$ for an arbitrary fixed $C,\beta \neq0$. 
    Then we have the MSA with $\mu=e^{\beta}$ and $\nu=1$. 
    
    Fourth, a straight line possesses the MSA with arbitrary $\mu,\nu\geq 0$. 

    Conversely, if a curve $\gamma(t)$ possesses the MSA, then there exist $\mu,\nu>0$ such that
    \begin{equation}\label{eqMSA2}
        \Lambda_\varepsilon( s(t), \rho(t))=(\mu^\varepsilon s(t) ,\nu^\varepsilon \rho(t)),\ 
    \end{equation}
    for any $\varepsilon>0$.
    Then, taking $\varepsilon$-differential of the first components of both sides of \eqref{eqMSA2} at $\varepsilon=0$ and applying \eqref{ALshift}, we have:
    \begin{align}
        &\lim_{\varepsilon\rightarrow0}\frac{\Lambda_\varepsilon s(t)-s(t)}{\varepsilon}=
        \lim_{\varepsilon\rightarrow0}\frac{s(t+\varepsilon)-s(\eta+\varepsilon)-s(t)+s(\eta)}{\varepsilon}=\dot s(t)-\dot s(0),
        \\
        &\lim_{\varepsilon\rightarrow0}\frac{\mu^\varepsilon-1}{\varepsilon} s(t)=s(t)\log\mu .
\end{align}
By solving $\dot s(t)-\dot s(0)= s(t)\log\mu $, we obtain 
\begin{align}
         s(t)&=
        \begin{cases}\smallskip
            \frac{\dot s_0}{\log \mu}(\mu^t-1) & \text{if $\mu\neq1$}, \\
            {\dot s_0}t & \text{if $\mu=1$}.
        \end{cases}
    \end{align}
The function $\rho(t)$ is determined by similar procedure from the second components of \eqref{eqMSA2} as
    \begin{align}
        &\lim_{\varepsilon\rightarrow0}\frac{\Lambda_\varepsilon \rho(t)-\rho(t)}{\varepsilon}=\lim_{\varepsilon\rightarrow0}\frac{\rho(t+\varepsilon)-\rho(t)}{\varepsilon}=\dot\rho (t),
        \\
        &\lim_{\varepsilon\rightarrow0}\frac{\nu^\varepsilon-1}{\varepsilon} \rho(t)=\rho(t)\log\nu,
\end{align}
so that we have
\begin{align}
         &\rho(t)=
        \begin{cases}
            \rho_0\nu^t & \text{if $\nu\neq1$},\\
             \rho_0 & \text{if $\nu=1$}.
        \end{cases}
    \end{align}
    Thus we obtain  
    \begin{equation}
        \rho(t)=\rho_0\nu^t=\rho_0\mu^{t\log_\mu\nu}=\left(\rho_0^{\log_\nu\mu}\frac{\log\mu}{\dot s_0}s(t)+\rho_0^{\log_\nu\mu}\right)^\frac{1}{\log_\nu\mu}.
    \end{equation}
    If $\nu\neq1$, and hence $\gamma$ is a LAC with $\alpha=\log_\nu\mu$.  
    Otherwise, we obtain a circle or a straight line by Proposition \ref{FTC}. 
\end{proof}
\subsection{The Harada self-affinity}\label{subsec:HSA}
Definition \ref{HSA_original} is formulated as follows. 
\begin{definition}\label{HSA}
    We say that a curve $\gamma(s):I\rightarrow \mathbb{C}$ posseses the \textit{the Harada self-affinity (the HSA)} if for any subinterval $J\subset I$ (homeomorphic to $I$), there exists a pair $(\sigma_J, F_J)$ of a reparametrization $\sigma_J:I\rightarrow J$ and an affine map $F_J:z\mapsto A_J z+b_J$ in $\mathrm{Aff}(\mathbb{C})$ such that 
    \begin{equation}\label{HSAeq}
        \gamma(\sigma_J(t))=F_J\gamma(t),\ \forall t\in I. 
    \end{equation}
\end{definition}
For geometric description of the HSA compared with the MSA, we refer to Remark \ref{geomSA}. 

\begin{remark}\label{remHSA}
    Let $\gamma(s)=x(s)+\imaginary y(s):I\rightarrow \mathbb{C}$ be a curve with the HSA. 
    Then, the following holds for any subinterval $J\subset I$. 
    \begin{enumerate}
    \item The arc length parameterization of the curve $F_J \gamma(\tau_J (s)):J\rightarrow\mathbb{C}$ is given by $\tau_J :=\sigma_J ^{-1}:J\rightarrow I$. 
    For two possible $\sigma_J=\sigma_J^{(1)},\sigma_J^{(2)}$, since both of their inverse functions are arc length parameters $F_J\gamma$, we have $(\sigma_J^{(1)})^{-1}(s)=\pm(\sigma_J^{(2)})^{-1}(s)+\eta$ for some $\eta\in\mathbb{R}$. 
    In this sense, $\sigma_J$ is uniquely determined by $J$. 
    \item 
    An affine map $F_J$ is unique up to set-wise automorphisms (not giving pointwise correspondence but curve-to-curve correspondence) in $\mathrm{Aut}(\gamma(I)):=\{G\in\mathrm{Aff}(\mathbb{C})\mid G\gamma(I)=\gamma(I)\text{ as sets}\}$. 
    We assume that $J\mapsto (\sigma_J,F_J)$ is well-defined modulo half-translations and $\mathrm{Aut}(\gamma(I))$. 
    \item For any $G\in \mathrm{Aff}(\mathbb{C})$, one can see that the curve $G\gamma$ posseses the HSA by replacing $F_J$ with $GF_JG^{-1}$. 
    Considering curves modulo $\mathrm{Aff}(\mathbb{C})$, the parameter $s$ does not work as an arc length parameter. 
    We use the variables $t,u=\sigma_J(t)$ to represent parameters in $I,J$, respectively.  
    Up to scaling and translation, we can regard $u$ lying on $[0,1]$ without loss of generality. 
    \item The affine map $F_J$ acts on the gradient $z(t):=\frac{dy}{dx}(t)=\frac{dy(t)/dt}{dx(t)/dt}$ by injective \textit{M\"obius transformation}
    \begin{equation}
        M_{A_J}(z)=\frac{a+bz}{c+dz},\ A_J=\begin{pmatrix}
        a& b\\ c&d
    \end{pmatrix}.
    \end{equation}
    If $\gamma(s)$ is not a line, there exists a point such that 

    \begin{equation}
        z'=\frac{x'y''-x''y'}{x'^2}=\frac{\kappa}{x'^2}
    \end{equation}
    does not vanish. 
    Thus $z$ is locally injective by the inverse function theorem, and should be so in the whole $I$ by the HSA. 
    In particular, if a curve has a nontrivial winding index (like LACs) for which $\theta(s)=\arctan z(s)$ is not injective, then it no longer possesses the HSA. 
    \end{enumerate}
\end{remark}
We now establish the following theorem. 
\begin{theorem}\label{HSAparabola}
A curve possesses the Harada self-affinity if and only if it is either a line or a parabola. 
\end{theorem}
    Figure \ref{fig:HSA} shows the HSA of lines and parabolas. 
    In the figure, each subcurve is the image of the affine map defined by the bounding parallelograms. 
    The bounding parallelogram of $y=x^2$, $a\leq x\leq b$ is spanned by $b-a+2\imaginary ab$ and $\imaginary (a+b)^2$.\textbf{}
\begin{figure}[h]\label{fig:HSA}
    \includegraphics[height=60mm]{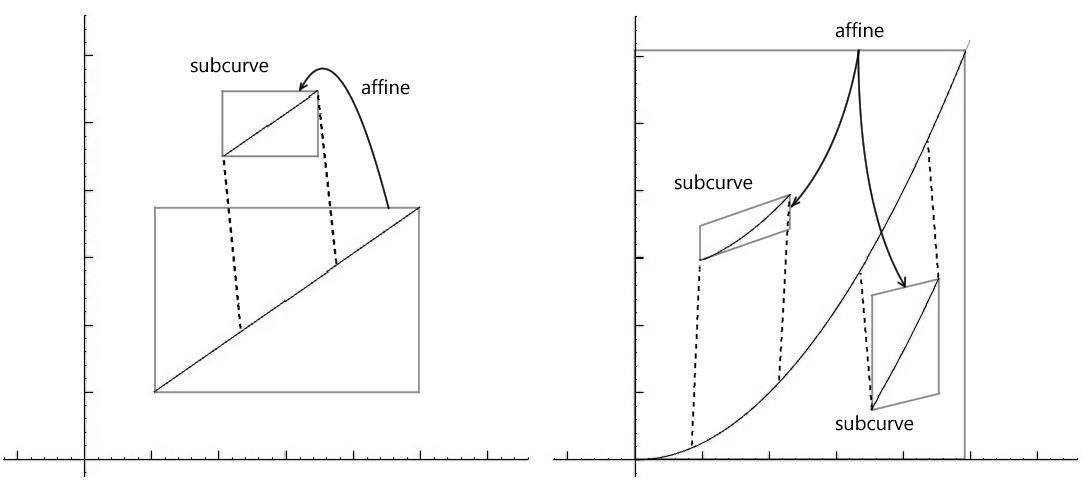}
    \caption{the Harada self-affinity. 
    }
\end{figure}

In order to prepare for the proof, we consider the following setting for a technical reason. 
We first separate the interval into two subintervals and observe equilibria for the affine transformation $F_J$ associated with subinterval $J$. 
\begin{definition}
    Let $\gamma(u):[0,1]\rightarrow \mathbb C$ possess HSA. 
    For any fixed $p \in I$, let $\check I :=[0,p ]\subset I$ and denote by $\check \sigma, \check \tau, \check F, \check A, \check b$ the corresponding ones $s_{\check I}, t_{\check I}, F_{\check I}, A_{\check I}, {b}_{I_p }$, respectively. 
    We define $\hat{I} :=[p ,1]$ and $\hat \sigma, \hat \tau, \hat F, \hat A, \hat b $ in the same way. 
\end{definition}
We note that an arbitrary subinterval $J=[a,b]$ is represented by 
\begin{equation}
    \sigma_{[0,b]}\circ\sigma_{[a',1]}([0,1])=[a,b],
\end{equation}
where $a'=\sigma_{[0,b]}^{-1}(a)$. 
We can deal with the HSA by considering the above setting without loss of generality. 

By definition \eqref{HSAeq}, we have
\begin{align}
    \gamma(\check \sigma  (t))=\check A \gamma(t)+\check b ,\quad\check \sigma  (0)=0,\ \check\sigma(1)=p,\label{eq1_sigma}
    \\
    \gamma(\hat\sigma  (t))=\hat A \gamma(t)+\hat b , \quad\hat\sigma (0)=p,\ \hat\sigma(1)=1\label{eq1^sigma}.
\end{align}
Substituting $t=0,1$ into \eqref{eq1_sigma} and \eqref{eq1^sigma} respectively, we have 
\begin{align}
    \gamma(0)=\check A \gamma(0)+\check b ,\quad\gamma(p )=\check A \gamma(1)+\check b ,\label{eq2_sigma}
    \\
    \gamma(p )=\hat A \gamma(0)+\hat b , \quad\gamma(1)=\hat A \gamma(1)+\hat b , \label{eq2^sigma}
\end{align}
which gives
\begin{align}
    \gamma(p )-\gamma(0)&=\check A (\gamma(1)-\gamma(0)),\label{eq3_sigma}
    \\
    \gamma(1)-\gamma(p )&=\hat A (\gamma(1)-\gamma(0)). \label{eq3^sigma}
\end{align}
\begin{lemma}\label{lemsymmetry}
The following hold. 
\begin{align}
    &(\check A -\hat A )\frac{\gamma(1)-\gamma(0)}{2}=\gamma(p )+\frac{\gamma(1)+\gamma(0)}{2},\label{eq4sigma}\\
    &(\check A +\hat A -I)\frac{\gamma(1)-\gamma(0)}{2}=0,\label{eq5sigma}\\
    &\check A \frac{\gamma(1)+\gamma(0)}{2}+\hat A \frac{\gamma(1)-\gamma(0)}{2}=-\check b ,\label{eq6_sigma}\\
    &\check A \frac{\gamma(1)-\gamma(0)}{2}+\hat A \frac{\gamma(1)+\gamma(0)}{2}=\hat b. \label{eq6^sigma}
\end{align}
\end{lemma}
\begin{proof} \eqref{eq4sigma} follows by subtracting \eqref{eq3^sigma} from \eqref{eq3_sigma}. 
We obtain \eqref{eq5sigma} by adding \eqref{eq3_sigma} to \eqref{eq3^sigma}. 
\eqref{eq6_sigma} and \eqref{eq6^sigma} follow from \eqref{eq2_sigma}, \eqref{eq3_sigma} and \eqref{eq2^sigma}, \eqref{eq3^sigma}, respectively. 
\end{proof}

\begin{corollary}\label{corsigma}
If $-\gamma(0)=\gamma(1)=1\in\mathbb{R}$, the following hold. 
    \begin{enumerate}
        \item $        (\check A -\hat A )\cdot1=\gamma(p )=\check b +\hat b $
        \item $(\check A +\hat A )\cdot1=1$
        \item $\check b =-\hat A \cdot1=\frac{1}{2}(\gamma(p )-1)$, $\hat b =\check A  \cdot1=\frac{1}{2}(\gamma(p )+1)$
    \end{enumerate}
\end{corollary}
\begin{proof}
    Substituting $\gamma(1)-\gamma(0)=2$, $\gamma(1)+\gamma(0)=0$ into Lemma \ref{lemsymmetry}, we get (1), (2), $\check b =-\hat A \cdot 1$ and $\hat b=\check A \cdot 1$. 
    By \eqref{eq1_sigma} and \eqref{eq1^sigma}, we have
    \begin{align}
        \gamma(p )+1&=[\gamma(\check \sigma  (t))]_{0}^{1}=[\check A \gamma(t)+\check b ]_{0}^{1}=2\check A  \cdot1,\\
        1-\gamma(p )&=[\gamma(\hat\sigma  (t))]_{0}^{1}=[\hat A \gamma(t)+\hat b ]_{0}^{1}=2\hat A  \cdot1.
    \end{align}
    This completes the proof. 
\end{proof}
Differentiating \eqref{eq1_sigma} and \eqref{eq1^sigma} by $t$, we have
\begin{align}\label{eqdt_sigma}
    \gamma'(\check \sigma  (t))\frac{\partial \check \sigma  (t)}{\partial t}&=\check A \gamma'(t),\\
    \gamma'(\hat\sigma  (t))\frac{\partial \hat\sigma  (t)}{\partial t}&=\hat A \gamma'(t). \label{eqdt^sigma}
\end{align}

\begin{proposition}\label{bdryrigidity}
    If a curve $\gamma(u):[0,1]\rightarrow\mathbb{C}$ posseses the HSA and either $\gamma'(0)$ or $\gamma'(1)$ is parallel to $ \gamma(1)-\gamma(0)$, $\gamma$ is a line segment whose image is $[-1,1]$. 
\end{proposition}
\begin{proof}
    We may assume that $-\gamma(0)=\gamma(1)=1\in\mathbb{R}$ modulo $\mathrm{Aff}(\mathbb C)$. 
    If $\gamma'(1)$ is parallel to $ \gamma(1)-\gamma(0)=2$, one can denote $\gamma'(1)=v\in\mathbb R$. 
    Then, by substituting $t=1$ and $\check \sigma  (1)=p $ to \eqref{eqdt_sigma}, we have
    \begin{align}
        \gamma'(p )\frac{\partial \check \sigma  (t)}{\partial t}\Big|_{t=1}&=\check A  \gamma'(1)=\check A  v=\frac{v}{2}(1+\gamma(p )) 
    \end{align}
    by Corollary \ref{corsigma}. 
    If $\frac{\partial \check \sigma  (t)}{\partial t}\Big|_{t=1}=0$, then we have $\gamma(p)=-1$. 
    Otherwise, taking gradients (the ratio of $x$- and $y$-coordinates) of both sides, we obtain 
    \begin{align}
        \frac{dy}{dx}(p)&=\frac{y(p)}{x(p)+1},
        \end{align}
        which implies 
        \begin{align}
        |y(p)|&=C|x(p)+1|,
    \end{align}
    where $\gamma(p )=x+\imaginary y=x(p )+\imaginary y(p )$ and $C$ is an arbitrary constant. 
    Substituting $p=1$, we get $C=0$. 
    Conversely, the line segment $[-1,1]$ obviously satisfies the assumption of the proposition. 
\end{proof}

Hereafter we assume that $-\gamma(0)=\gamma(1)=1$, $\gamma'(1)=\imaginary$ modulo $\mathrm{Aff}(\mathbb{C})$. 
By Corollary \ref{corsigma}, we may define $\check\alpha(p) ,\check\beta(p) ,\hat\alpha(p) ,\hat\beta(p) \in\mathbb{R}$ so that 
\begin{align}
    \check A =\begin{pmatrix}\frac{1}{2}(1+x(p )) &\check\alpha(p) \\\frac{1}{2}y(p )&\check\beta(p)  \end{pmatrix},\ \hat A= \begin{pmatrix}\frac{1}{2}(1-x(p )) &\hat\alpha(p) \\-\frac{1}{2}y(p )&\hat\beta(p)  \end{pmatrix}. 
\end{align}
Taking gradients of each sides of \eqref{eqdt_sigma} and \eqref{eqdt^sigma}, denoting $z(t)=\frac{dy}{dx}(t)$, we have
\begin{align}
    z(\check \sigma  (t))&=\frac{y(p )x'(t)+2\check\beta (p)  y'(t)}{(1+x(p )) x'(t)+2\check\alpha (p)  y'(t)}=M_{\check A }z(t),\\ 
    z(\hat\sigma  (t))&=\frac{-y(p )x'(t)+2\hat\beta (p)  y'(t)}{(1-x(p )) x'(t)+2\hat\alpha (p)  y'(t)}=M_{\hat A }z(t).
\end{align}
Substituting $t=0,1$, we have
\begin{align}
    z_0:=z(0)&=M_{\check A }z(0)=\frac{y(p)+2\check\beta (p)  z_0}{1+x(p)+2\check\alpha (p)  z_0}, \label{eqz_0}\\
    z(p )&=M_{\check A }z(1)=\frac{y(p)x'(0)+2\check\beta (p)  y'(0)}{(1+x(p))x'(0)+2\check\alpha (p)  y'(0)}=\frac{\check\beta (p) }{\check\alpha (p) }, \label{eqz_1}\\
    z(p )&=M_{\hat A }z(0)=\frac{-y(p)+2\hat\beta (p)  z_0}{1-x(p)+2\hat\alpha (p)  z_0}. \label{eqz^0}
\end{align}
In addition, we have
\begin{align}
    z(1)&=M_{\hat A }z(1)=\frac{-y(p)x'(0)+2\hat\beta (p)  y'(0)}{(1-x(p))x'(0)+2\hat\alpha (p)  y'(0)}=\frac{\hat\beta (p) }{\hat\alpha (p) },\label{eqz^1}
\end{align}
which implies $\hat\alpha (p) =0$ by $\gamma'(1)=\imaginary$.

\begin{proof}[Proof of Theorem \ref{HSAparabola}]
    Suppose that $\gamma$ possesses the HSA and is not a line. 
    Proposition \ref{deformation_circle} yields
    \begin{align}
        \kappa_\gamma(\check \sigma  (t))&=\frac{\mathrm{det}\check A }{|\check A \gamma'(t)|^3}\kappa_\gamma(t),
        \label{eqkappa_sigma}
        \\
        \kappa_\gamma(\hat\sigma  (t))&=\frac{\mathrm{det}\hat A }{|\hat A \gamma'(t)|^3}\kappa_\gamma(t).
        \label{eqkappa^sigma}
    \end{align}
    Substituting $t=0,1$ into \eqref{eqkappa_sigma} and \eqref{eqkappa^sigma}, we have
    \begin{align}
        \label{eqkappa_sigma0}
        \kappa_\gamma(0)&=\frac{\mathrm{det}\check A }{|\check A \gamma'(0)|^3}\kappa_\gamma(0)
        ,\\
        \label{eqkappa_sigma1}
        \kappa_\gamma(p )&=\frac{\mathrm{det}\check A }{|\check A \gamma'(1)|^3}\kappa_\gamma(1)=\frac{\check\beta (p)  (1+x(p ))-\check\alpha (p)  y(p )}{2|\check\alpha (p) +\imaginary\check\beta (p) |^3}\kappa(1),\\
        \label{eqkappa^sigma0}
        \kappa_\gamma(p )&=\frac{\mathrm{det}\hat A }{|\hat A \gamma'(0)|^3}\kappa_\gamma(0)
        ,
        \\
        \label{eqkappa^sigma1}
        \kappa_\gamma(1)&=\frac{\mathrm{det}\hat A }{|\hat A \gamma'(1)|^3}\kappa_\gamma(1)=
        \frac{\hat\beta (p) (1-x(p ))}{2|\hat\beta (p) |^3}\kappa(1).
    \end{align}
    Here we have used $\hat{\alpha}(p)=0$. 
    First, we show that none of the following occurs: 
    \begin{enumerate}
        \item $\check A \not \in GL(2,\mathbb{R})$ for any $p \in I$.
        \item $\hat A \not \in GL(2,\mathbb{R})$ (equivalently $\hat\beta (p) =0$ or $x(p )=1$) for any $p \in I$.
        \item $\kappa_\gamma(0)=0$.
        \item $\kappa_\gamma(1)=0$.
    \end{enumerate}
    We show that any of the above implies that $\gamma$ is a line segment, which contradicts 
    the assumption. 
        It follows from (1) that $z(p)={\check\beta (p) }/{\check\alpha (p) }={y(p)}/{(x(p)+1)}$, and the discussion in the proof of Proposition \ref{bdryrigidity} works. 
    If (2) holds, then \eqref{eqdt^sigma} at $t=1$ implies that $\gamma'(p )=0$ for any $p$.   
    If (3) ((4), respectively) holds, then \eqref{eqkappa^sigma0} (\eqref{eqkappa_sigma1}, respectively) implies that $\kappa_\gamma(p )=0$ unless (1) ((2), respectively).

    Next, it follows from \eqref{eqkappa^sigma1} and $\kappa_\gamma(1)\neq0$ that 
    \begin{align}\label{eqkappa1ratio}
        \frac{\mathrm{det}\hat A }{|\hat A \gamma'(1)|^3}&=\frac{\hat\beta (p) (1-x(p ))}{2|\hat\beta (p) |^3}=1,
    \end{align}
    which yields
    \begin{align}\label{hatbeta}
        \hat\beta (p) &=\beta(x(p )):=\begin{cases}\medskip
            -\sqrt{\frac{1}{2}{(x(p )-1)}}&\text{ if }x(p )>1,\\
            \sqrt{\frac{1}{2}({1-x(p )})}&\text{ if }x(p )\leq 1.
        \end{cases}
    \end{align}
    Note that we used the fact that sign$(\hat\beta (p) )$ equals to sign$(1-x(p ))$ which follows from \eqref{eqkappa1ratio}.
    Compared with \eqref{eqz^0}, if $x(p)\neq 1$ we have 
    \begin{align}
        z(p )=\frac{-y(p)+2z_0\hat\beta (p) }{1-x(p)},
    \end{align}
    or
    \begin{align}
        \frac{dy}{dx}+\frac{y}{1-x}=\frac{2z_0\beta(x)}{1-x}. 
        \label{diffeq}
    \end{align}
    Solving \eqref{diffeq} by a standard technique, we obtain 
    \begin{align}\label{affparabola}
        y(x)&=\begin{cases}\medskip
            {-2\sqrt{2}z_0}{\sqrt{x-1}}+C_+(x-1)&\text{ if }x(p )>1,\\
            {2\sqrt{2}z_0}{\sqrt{1-x}}+C_-(1-x)&\text{ if }x(p )<1,
        \end{cases}     
    \end{align}
    where $C_+,C_-$ are arbitrary constants. Figure \ref{fig:graphHSA} shows the graph of \eqref{affparabola}. 
        \begin{figure}[h]\label{fig:graphHSA}
        \includegraphics[height=80mm]{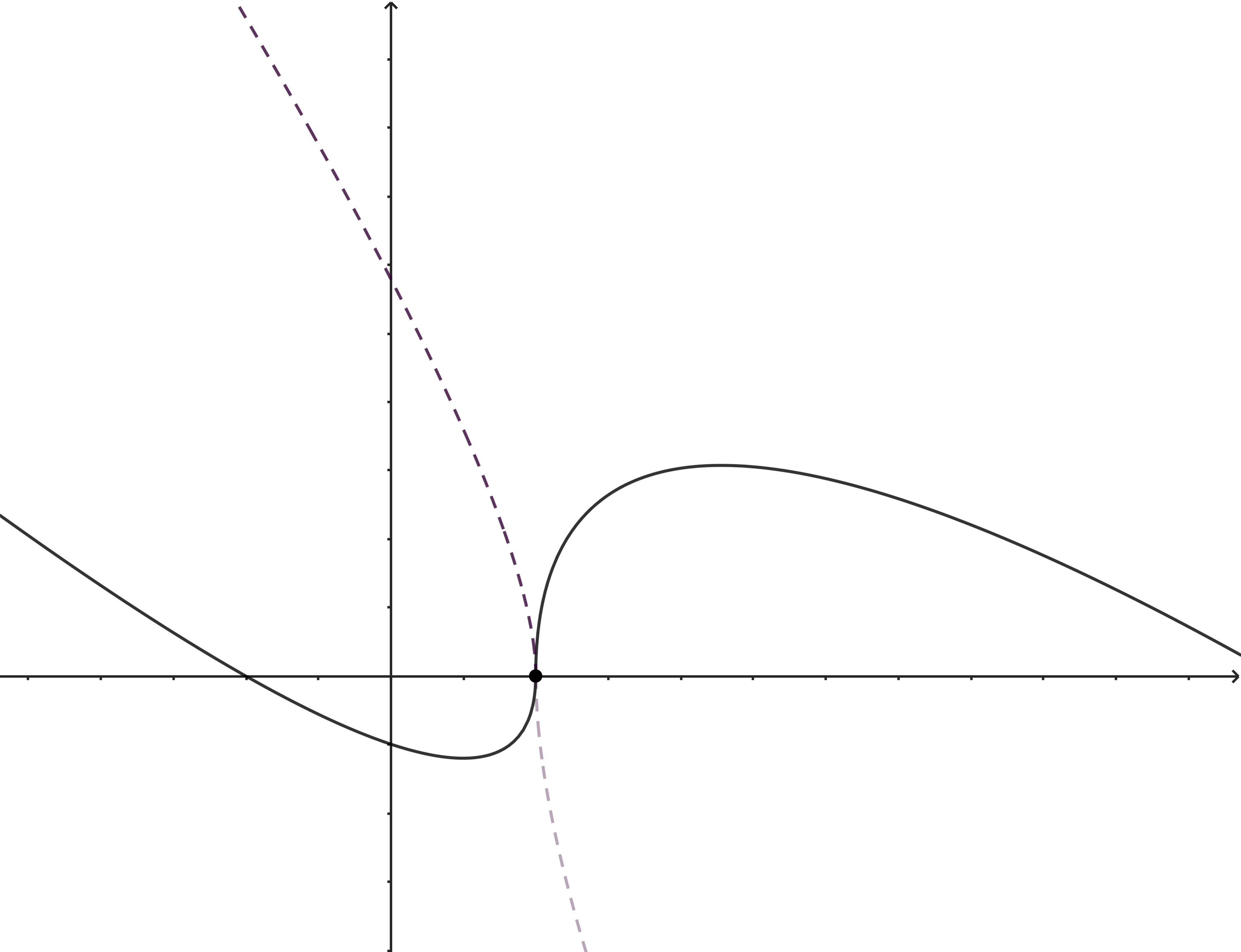}
        \caption{the parabolas obtained from the HSA
        }
        \end{figure}  
    We remark that these two parabolas arise from \eqref{hatbeta} separately, and so that each parabola possesses the HSA independently. 
    The combined curve does not possess the HSA. 
    For any isolated $p$ with $x(p)=1$, we have $\gamma(p)=(1,0)$ by taking limits of $\eqref{affparabola}$ as $x\rightarrow1$. 
    Thus a curve with the HSA should be either a line or the parabola up to affine deformations. 
    

    Conversely, the parabola $P(t)=(t,t^2)$ $(0\leq t \leq1)$ posseses the HSA.
    In fact, for any $0\leq t_0\leq t_1\leq 1$, consider 
    \begin{align}
    \sigma(t):=(t_1-t_0)t+t_0,\ t\in[t_0,t_1].  
    \end{align}
    Then, $\sigma=\sigma(t)$ runs monotonously through $[t_0, t_1]$.
    The trivial formula 
    \begin{align}\label{HSAparabolaeq}
        \begin{pmatrix}
            \sigma\\\sigma^2
        \end{pmatrix}&=
        \begin{pmatrix}
            t_1-t_0&0\\2t_0(t_1-t_0)&(t_1-t_0)^2
        \end{pmatrix}
        \begin{pmatrix}
            t\\ t^2
        \end{pmatrix}+
        \begin{pmatrix}
            t_0\\ t_0^2
        \end{pmatrix}
    \end{align}
    yields that the curve $P_{[t_0,t_1]}(\sigma)=(\sigma,\sigma^2),\ \sigma\in[t_0,t_1]$ is a subcurve of $P(t)$ and an affine deformation of $P(t)$. Thus the parabola $P$ possesses the HSA. 

    The description of parabolas as quadratic curves \cite{CRCtable} shows that an arbitrary parabola is an affine deformation of the parabola $P$, so that it possesses the HSA by Remark \ref{remHSA}. 
    The claim for lines is obvious. 
    This completes the proof. 
\end{proof}

\section{Extendable self-affinity}

The result in Section \ref{subsec:HSA} may suggest that the other quadratic curves are characterized by some sort of self-affinity. 
We will discuss this point in the following. 
We first introduce a new self-affinity that generalizes the MSA and the HSA. 
\begin{definition}
    A curve $\gamma(s):[0,s_{\mathrm{all}}]\rightarrow\mathbb{C}$ possesses the \textit{extendable self-affinity (the ESA)} with respect to  a reparametrization $t=t(s)$ and a Lie group $G$ if there exists a supercurve $\tilde{\gamma}(s):[0,\tilde s_{\mathrm{all}}]\rightarrow\mathbb{C}$ of $\gamma$ and a differentiable map $F_\varepsilon:\mathbb{R}\rightarrow G$ such that for any $t,\varepsilon$ with $t, t+\varepsilon\in t([0,\tilde s_{\mathrm{all}}])$,
\begin{equation}
    \tilde\gamma(t+\varepsilon)=F_\varepsilon\tilde\gamma(t).
\end{equation}
\end{definition}
The MSA \eqref{MSA} can be regarded as the ESA with respect to the group of transition maps between the original curve and the curve whose curvature and line element are deformed by $(\kappa(t),ds(t))\mapsto (a\kappa(t),b \,ds(t))$, $a,b>0$. 
In fact, the transition map is given by the collection of homeomorphisms on $\mathbb{C}$ of the form $p \mapsto\phi_{a,b}\placket{ b\phi_{1,1}^{-1}(p)}$, where 
\begin{align}
    \phi_{a,b}(s)=\int_{0}^{\frac{s}{b}}\exp\placket{\imaginary\int_{0}^{\frac{s}{b}}a\kappa\placket{\frac{s}{b}}ds}ds:[0,bs_{\mathrm{all}}]\rightarrow \mathbb{C}.
\end{align}

The HSA \eqref{HSAeq} is the ESA with respect to a subgroup in $\mathrm{Aff}(\mathbb{C})$.
In the latter part of the proof of Theorem \ref{HSAparabola}, the fact that the parabola segment $P_{[t_0,t_1]}(\sigma)$ is an affine deformation of the parabola $P(t)$ is true not only for $0\leq t_0<t_1 \leq 1$ but for arbitrary $t_0<t_1$. 
In other words, we can define a unique extension of $P(t)$, $t\in[0,1]$ under the HSA.

Next, as a generalization of the HSA, we prove that the ESA characterizes the constant curvature curves in the \textit{equiaffine geometry} \cite{Equiaffine}. 
We say that a reparametrization $t$ of a curve is \textit{equiaffine parameterization} if $\det(\gamma_t,\gamma_{tt})=1$. 

We introduce the equiaffine frame $\Phi^{\mathrm{SA}}:=(\gamma_t,\gamma_{tt})\in SL(2,\mathbb{R})$ of an equiaffine parametrized curve $\gamma(t)$. 
The fundamental theorem of curves in equiaffine geometry \cite{Equiaffine} states that 
for a given non-negative, smooth function $\kappa^\mathrm{SA}(t):I\rightarrow \mathbb{C}$, 
the \textit{equiaffine Frenet formula} 
\begin{align}\label{SAFrenet}
    \Phi_t^{\mathrm{SA}}&=\Phi^{\mathrm{SA}}\begin{pmatrix}
        0&-\kappa^{\mathrm{SA}}\\1&0
    \end{pmatrix}
\end{align}
has a unique solution $\gamma(t):I\rightarrow \mathbb{C}$ 
 up to the \textit{equiaffine transformation group} $G^{\mathrm{SA}}:=\{z\mapsto Az+b\mid A\in SL(2,\mathbb{R}),\ b\in\mathbb{C}\}$. 

\begin{theorem}\label{extSA}
    A curve possesses the ESA with respect to the equiaffine parameterization and the equiaffine transformation group $G^{\mathrm{SA}}$ if and only if it is either a parabola, an ellipse, or a hyperbola. 
\end{theorem}
\begin{proof}
Let a curve $\gamma$ possess the ESA with respect to $G^{\mathrm{SA}}$. 
Then, there exists a reparametrization $t$ and $F_\varepsilon:\mathbb{R}\rightarrow G^{\mathrm{SA}}$ such that for any $t,\varepsilon$,
\begin{align}
    \gamma(t+\varepsilon)=F_\varepsilon\gamma(t).
\end{align}
Differentiating by $t$, we have
\begin{align}
    \gamma_t(t+\varepsilon)=F_\varepsilon\gamma_t(t).
\end{align}
Taking $\varepsilon$-differentials at $\varepsilon=0$, we have the following: 
\begin{align}
    \lim_{\varepsilon\rightarrow0}\placket{ \frac{\gamma(t+\varepsilon)-\gamma(t)}{\varepsilon},\frac{\gamma_t(t+\varepsilon)-\gamma_t(t)}{\varepsilon}}&=\lim_{\varepsilon\rightarrow0}\frac{F_\varepsilon-I}{\varepsilon}\Phi^{\mathrm{SA}}(t),
    \end{align}
which implies that
\begin{align}
    \Phi^{\mathrm{SA}}_t&=F\Phi^{\mathrm{SA}},
\end{align}
where $F=\frac{d}{d\varepsilon}F_\varepsilon\mid_{\varepsilon=0}$. 
Compared with \eqref{SAFrenet}, we have 
\begin{align}\label{ESA}
    \Phi_t=\Phi\begin{pmatrix}
        0&-\kappa^{\mathrm{SA}}\\1&0
        \end{pmatrix}&=F\Phi
\end{align}
Differentiating by $t$ and applying \eqref{ESA} yields 
\begin{align}
    \Phi\begin{pmatrix}
        0&-\kappa^{\mathrm{SA}}_t\\0&0
        \end{pmatrix}+F\Phi\begin{pmatrix}
        0&-\kappa^{\mathrm{SA}}\\1&0
        \end{pmatrix}&=F\Phi\begin{pmatrix}
        0&-\kappa^{\mathrm{SA}}\\1&0
        \end{pmatrix},
\end{align}
from which we obtain 
\begin{align}
    \Phi \begin{pmatrix}
        0&-\kappa^{\mathrm{SA}}_t\\0&0
        \end{pmatrix}&=0. 
\end{align}
By the assumption of equiaffine parameterization $t$ that $\Phi$ is regular, we conclude that the equiaffine curvature $\kappa^{\mathrm{SA}}$ is constant. 
By solving \eqref{SAFrenet} for constant $\kappa^{\mathrm{SA}}$, we see that $\gamma(t)$ is either a parabola ($\kappa^{\mathrm{SA}}=0$), an ellipse ($\kappa^{\mathrm{SA}}>0$), or a hyperbola ($\kappa^{\mathrm{SA}}<0$). 
    
Conversely, it follows from the addition theorem that 
\begin{align}\label{SAE}
        \begin{pmatrix}
            A\cos(t+\varepsilon)\\B\sin(t+\varepsilon)
        \end{pmatrix}&=
        \begin{pmatrix}
            \cos\varepsilon&-\frac{A}{B}\sin\varepsilon\\\frac{B}{A}\sin\varepsilon&\cos\varepsilon
        \end{pmatrix}
        \begin{pmatrix}
            A\cos t\\B\sin t
        \end{pmatrix}, \\
        \begin{pmatrix}
            A\cosh(t+\varepsilon)\\B\sinh(t+\varepsilon)
        \end{pmatrix}&=
        \begin{pmatrix}
            \cosh\varepsilon&\frac{A}{B}\sinh\varepsilon\\\frac{B}{A}\sinh\varepsilon&\cosh \varepsilon
        \end{pmatrix}
        \begin{pmatrix}
            A\cos t\\B\sin t
        \end{pmatrix}, 
    \end{align}
which implies the ESA of ellipses and hyperbolas, respectively. 
Together with Theorem \ref{HSAparabola}, the above completes the proof. 
\end{proof}

We note that Theorem \ref{MSALAC} in the case $\alpha=1$ implies that the ESA characterizes logarithmic spirals. 
They are the constant curvature curves in the \textit{similarity geometry} \cite{Inoguchi2018, Inoguchi2023}. 
Therefore, the results in this paper may be summarized as follows:
the constant curvature curves in similarity geometry and equiaffine geometry are captured by the common self-affinity as shown in Table \ref{table}. 

\begin{table}[htb]
\begin{center}
\begin{tabular}{|c|c|c|c|c|c|c|} \hline
   geometry& \multicolumn{3}{|c|}{similarity geometry}& \multicolumn{3}{|c|}{equiaffine geometry}   \\ \hline
   curvature& \centerbox{12mm}{0}  &\centerbox{12mm}{+}& \centerbox{12mm}{-}& \centerbox{12mm}{0} &\centerbox{12mm}{+}& \centerbox{12mm}{-}   \\ \hline
   curve& circle  &\multicolumn{2}{|c|}{logarithmic spiral}& parabola&ellipse& hyperbola   \\ \hline
   \multirow{2}{*}{self-affinity}& \multicolumn{3}{|c|}{MSA}& HSA &\multicolumn{2}{|c|}{}   \\ \cline{2-7}
   & \multicolumn{6}{|c|}{ESA}   \\ \hline
\end{tabular}
\caption{charactrization of constant curvature curves in similarity geometry and equiaffine geometry by self-affinities: the MSA and the HSA are included in the ESA}\label{table}
\end{center}
\end{table} 

\section{Concluding remarks}
In this paper, we have given rigorous definitions of the HSA and the MSA for planar curves, which have been proposed as properties to characterize curves that car designers regard as \textit{aesthetic}. 
Then, we have proved that 
\begin{itemize}
    \item a curve with the MSA is either a line, a circle, or a LAC,
    \item a curve with the HSA is either a line or a parabola, 
    \item a curve with the ESA in equiaffine geometry is either a parabola, an ellipse, or a hyperbola. 
\end{itemize}
With the notion of the ESA, the first two results intersect by one statement that a curve with the ESA in similarity geometry or equiaffine geometry has constant curvature. 

We intend to find a generalization of LACs to spatial curves and surfaces that reflects several properties of planar LACs. 
In addition to the MSA, LACs are known to have two other characterizations related to geometric shape generation. 
It is shown in \cite{Inoguchi2018,Inoguchi2023} that LACs are formulated by a variational principle and an integrable evolution in similarity geometry. 
Though there are other generalizations of the MSA to spatial curves \cite{Gobi20142} and surfaces \cite{Miura2014}, relations to the above characterizations are yet to be discussed.

The observation in this paper may imply that an alternate class of \textit{``aesthetic''} curves different from LACs can be captured in equiaffine geometry via self-affinity. 
We aim to give further investigations in the forthcoming paper.

\subsection*{Acknowledgments}
    This work was supported by JST CREST Grant Number JPMJCR1911. 
    The authors would be grateful to Prof.\ Jun-ichi Inoguchi, Prof.\ Gobithaasan Rudrusamy, Prof.\ Kenjiro T.\ Miura, and Prof.\ Yoshiki Jikumaru for their insightful comments and continuous encouragement. 
\section*{Declarations}
\subsection*{Funding}
    This work was supported by JST CREST Grant Number JPMJCR1911.
\subsection*{Conflict of interest}
    The authors declare that there is no conflict of interest.
\subsection*{Ethics approval}
    Not applicable.
\subsection*{Consent to participate}
    Not applicable.
\subsection*{Consent for publication}
    The authors agree to publish this work.
\subsection*{Authors’ contribution}
    Shun Kumagai developed theoretical formalism and wrote the original draft.
    Kenji Kajiwara was in charge of overall direction, conceptualization, and editing.
    Both authors discussed and wrote the final manuscript.
\subsection*{Data availability}
    No data was used for the research described in the article. 
\subsection*{Code availability}
    Not applicable. 
\bibliographystyle{ieeetr}
\bibliography{reference.bib} 

\end{document}